\documentclass{article}
\usepackage{amssymb, amsmath, amsthm}
\usepackage{wrapfig}

\DeclareMathOperator{\GL}{GL}
\DeclareMathOperator{\AGL}{AGL}
\DeclareMathOperator{\SL}{SL}
\DeclareMathOperator{\ULT}{ULT}
\DeclareMathOperator{\centraliser}{C}
\DeclareMathOperator{\class}{cl}
\DeclareMathOperator{\groupcentre}{\mathcal{Z}}
\DeclareMathOperator{\commutinggraph}{\Gamma}
\DeclareMathOperator{\distance}{d}
\DeclareMathOperator{\diameter}{diam}

\newcommand{\Wr}{\,\mathrm{wr}\,}
\newcommand{\cent}[2]{\centraliser_{#1}(#2)}
\newcommand{\cl}[2]{\class_{#1}(#2)}
\newcommand{\Z}{\mathbb{Z}}
\newcommand{\gcentre}[1]{\groupcentre(#1)}
\newcommand{\cg}[1]{\commutinggraph(#1)}
\newcommand{\dist}[2]{\distance(#1,#2)}
\newcommand{\diam}[1]{\diameter(#1)}
\newcommand{\W}{\mathcal{W}}
\newcommand{\mS}{\mathcal{S}}

\newcommand{\mQ}{\mathcal{Q}}
\newcommand{\ba}{\mathbf{a}}
\newcommand{\bbb}{\mathbf{b}}
\newcommand{\bx}{\mathbf{x}}
\newcommand{\by}{\mathbf{y}}
\newcommand{\al}{\alpha}
\newcommand{\be}{\beta}

\renewcommand{\leq}{\leqslant}
\renewcommand{\geq}{\geqslant}

\newtheorem{theorem}{Theorem}[section]
\newtheorem{lemma}[theorem]{Lemma}
\newtheorem{proposition}[theorem]{Proposition}
\newtheorem{corollary}[theorem]{Corollary}
\theoremstyle{plain}
\newtheorem{construction}[theorem]{Construction}

\title{On Bounding the Diameter of the Commuting Graph of a Group}
\author{Michael Giudici and Aedan Pope\footnote{This work was completed while the second author was an honours student at the University of Western Australia. This work forms part of a ARC Discovery Project held by the first author.} \\
        School of Mathematics and Statistics\\ 
        The University of Western Australia\\ 
        35 Stirling Highway\\
		Crawley WA 6009\\
		michael.giudici@uwa.edu.au, aedanpope@gmail.com}
\date{}

\begin{document}

\maketitle

\begin{abstract}
The commuting graph of a group $G$ is the simple undirected graph whose vertices are the non-central elements of $G$ and two distinct vertices are adjacent if and only if they commute. It is conjectured by Jafarzadeh and Iranmanesh that there is a universal upper bound on the diameter of the commuting graphs of finite groups when the commuting graph is connected. In this paper we determine upper bounds on the diameter of the commuting graph for some classes of groups to rule them out as possible counterexamples to this conjecture. We also give an example of an infinite family of groups with trivial centre and diameter $6$, the previously largest known diameter for an infinite family was $5$ for $S_n$.
\end{abstract}

\section{Introduction}

For a group $G$, we denote the \emph{center} of $G$ by $\gcentre{G}$ and $\gcentre{G}=\{x\in G \mid xy=yx~\forall y \in G \}$. The \emph{commuting graph} of a group, denoted by $\cg{G}$, is the simple undirected graph whose vertices are the non-central elements of $G$ and two distinct vertices $x$ and $y$ are adjacent if and only if $xy=yx$. In particular, the set of neighbours of $x$ is the set of all non-central elements of the \emph{centraliser} of $x$ in $G$, that is, of $\cent{G}{x}=\{g\in G\mid xg=gx\}$.

A \emph{path} in a graph is an ordered list $a_1, a_2, \ldots , a_k$ of vertices where there is an edge in the graph from $a_i$ to $a_{i+1}$ for all $i$; the path is said to between $a_1$ and $a_k$ and of length $k-1$. A graph is \emph{connected} if and only if there exists a path between any two distinct vertices in the graph. The \emph{distance} between two vertices of a graph, say $x$ and $y$, is the length of the shortest path between $x$ and $y$ in the graph if such a path exists and is $\infty$ otherwise; this is denoted $\operatorname{d}(x,y)$. The \emph{diameter} of a graph $\Gamma$ is the maximum distance between any two vertices in the graph, and is denoted $\diam{\Gamma}=\max\{\operatorname{d}(x,y)\mid x,y \in \Gamma\}$.

Commuting graphs were first studied by Brauer and Fowler in 1955 \cite{brauer} to prove results fundamental to the Classification of Finite Simple Groups. Further applications include \cite{fischer1971,segev2002}. In 2002 Segev and Seitz \cite{segev2002} began the investigation of commuting graphs in their own right by proving that, for all finite simple classical groups $G$ over a field of size at least $5$, the diameter of $\cg{G}$ is at most 10 when $\cg{G}$ is connected. It is not known how sharp this bound is. Iranmanesh and Jafarzadeh \cite{iranmanesh2008} continue this investigation and determine the conditions for the commuting graph of a symmetric or alternating group to be connected and that the diameter is at most $5$ in these cases. They conjecture that there is a universal upper bound on the diameter of a connected commuting graph for any finite nonabelian group.

This unresolved conjecture is our primary motivation, and has more supporting evidence. The group of all invertible matrices with size at least $3$ over a field of size at least $3$ has commuting graph with diameter between $4$ and $6$ when connected, and there is an upper bound on the diameter of the group of those with determinant $1$ \cite{akbari2006}.  Again, it is not known if these bounds are sharp. The group of all invertible matrices with size at least $2$ over the integers modulo some composite positive integer has commuting graph with diameter exactly $3$ \cite{giudici2010}. If $G$ is a non-trivial finite solvable group with trivial centre and no cyclic or generalised quaternion Sylow subgroups then $\cg{G}$ is connected with diameter at most $7$ \cite{woodcock2010}. The situation is completely different for semigroups, with Ara\'ujo, Kinyon and Konieczny \cite{araujo} proving that for all positive integers $n$, there is a semigroup $S$ such that the diameter of $\cg{S}$ is $n$.

We prove more theorems pertaining to this conjecture, beginning with the commuting graphs that result from building up groups out of smaller groups. We denote the group of all permutations on a set of size $n$ by $S_n$, and the group of all even permutations on such a set by $A_n$. For any $H\leqslant S_n$ we denote the wreath product of a group $G$ with $H$ by $G \Wr H=G^n\rtimes H$.
 
\begin{theorem}
\label{pWreathProduct}
Consider $A \Wr S_n$ for some positive integer $n \geq 2$ and finite group $A$ of even order and trivial centre. If for every prime $p \neq 2$ dividing the order of $A$, the number of conjugacy classes of elements of order $p$ in $A$ is less than $n$, then the commuting graph of $A \Wr S_n$ is connected with diameter at most $7$.
\end{theorem}

We do not know how sharp the bound is in Theorem \ref{pWreathProduct}. Computer calculations on small groups reveal that $\cg{G}$ for $G=S_3\Wr S_2$, $S_3\Wr S_3$, $S_4\Wr S_2$, $A_4\Wr S_3$ and $D_{10}\Wr S_3$ all have diameter 4 while $\cg{G}$ for $G=D_{18}\Wr S_2$ and $A_5\Wr S_2$ have diameter 5.

\begin{theorem}
\label{tCentralProduct}
Suppose a non-abelian finite group $G$ is the central product of two of its subgroups $H$ and $K$. If $H$ and $K$ are both non-abelian then
$$\diam{\cg{G}} \leq \min\{3,\diam{\cg{H}}, \diam{\cg{K}}\}.$$ Otherwise, exactly one of $H$ and $K$, say $K$, is abelian and then $$\diam{\cg{G}} = \diam{\cg{H}}.$$
\end{theorem}
Theorem \ref{tCentralProduct} leads to the following corollary; that if the conjecture of Jafarzadeh and Iranmanesh is true for $p$-groups then it is true for nilpotent groups:
\begin{corollary}
\label{cNilpotentToPrime}
Let $G$ be a finite nilpotent group. Then either $\diam{\cg{G}}\leq 3$ or $G$ is a $p$-group.
\end{corollary}

The following results depend on comparing the size of the centre with the size of the whole group. The derived subgroup of a group $G$ is denoted by $G'$ and is the subgroup generated by all commutators $[g,h]=g^{-1}h^{-1}gh$ of the elements $g,h\in G$. This group $G'$ is the smallest normal subgroup $N$ of $G$ such that $G/N$ is abelian.
\begin{theorem}
\label{tSmallCentre}
If $G' \leq \gcentre{G}$ and $|\gcentre{G}| ^ 3 < |G|$ then $\diam{\cg{G}} = 2$.
\end{theorem} 

Theorem \ref{tSmallCentre} covers groups with nilpotency class $2$ and sufficiently small centres, in particular it includes all extraspecial $p$-groups except those with order $p^3$. These remaining extraspecial $p$-groups are dealt with by the next result, where we look at groups with relatively large centres.
\begin{theorem}
\label{tPrimeCentre}
If the index of the centre of a group $G$ is finite and a product of at most $3$ primes, not necessarily distinct, then the commuting graph of $G$ is disconnected. 
\end{theorem}

Theorem \ref{tPrimeCentre} allows us to show that many small groups have disconnected commuting graphs. For example, since $p$-groups have non-trivial centres, any $p$-group with order at most $p^4$ has disconnected commuting graph. By Theorem \ref{tSmallCentre}, extraspecial $p$-groups of order $p^5$ have commuting graph with diameter $2$ and so  the condition of up to $3$ primes in Theorem \ref{tSmallCentre} is tight.

In Section \ref{sMotivatingExamples}, we provide some illuminating examples to motivate future work on the Iranmanesh and Jafarzadeh conjecture.  For all primes $p$ and integers $n \geq 2$, the group of lower unitriangular matrices of size $n$ over $\Z_p$, denoted $\ULT(n,p)$, is a $p$-group with nilpotency class $n-1$ \cite[Theorem 3.2.3]{weinstein}. Theorem \ref{tSmallCentre} deals with groups with nilpotency class $2$, in contrast we will see in Proposition \ref{pUnitriangular} that $\ULT(n,p)$ is a family of groups with arbitrarily large nilpotency class but fixed commuting graph diameter $3$.

The largest known theoretical result of the diameter of a commuting graph is $5$ for $S_n$ when $n$ and $n-1$ are both not prime \cite{woodcock2010}. Groups with commuting graph with a diameter of $6$ were found in a computer search in \cite{giudici2010}, but this offers little insight into why such a high diameter holds, and what properties of the group are important. Here we will build an infinite family of groups as a semidirect product of $\Z_p^2$ with $\SL(2,3)$, for certain primes $p$, that have connected commuting graphs with diameter $6$. This is the first infinite family of groups with commuting graph diameter $6$ known to the authors.. The family of groups we construct all have trivial centre, and the proof gives us some idea of the relevant properties of the group.

\section{The diameter of the commuting graph of some group products}
For all $n \geq 2$ the induced subgraph of $\cg{A_5 \Wr S_n}$ whose vertices are all the elements of a particular conjugacy class of involutions  has diameter $n$ \cite{bates2004}. Here we prove Theorem \ref{pWreathProduct}. Since $A_5$ has one conjugacy class of elements of order $3$ and two classes of elements of order $5$ this implies that $\cg{A_5 \Wr S_n}$ for $n\geq 3$ has diameter at most $7$. We are not sure how sharp this bound is. A computer calculation shows that the diameter of $\cg{A_5\Wr S_2}$ is 5. 

We start with a lemma connecting elements of prime order to involutions in the commuting graph.
\begin{lemma}
\label{lPrimeOrderInvolutions}
Let $A \Wr S_n$ satisfy the conditions of Proposition \ref{pWreathProduct}, then in this group every element of prime order commutes with some involution.
\end{lemma}
\begin{proof}
Set $G := A \Wr S_n$ and let $g \in G$ be an element with prime order $p$.  Since an element commutes with itself we may assume that  $p$ is odd. We write $g=(a_1,a_2,\ldots,a_n)\pi$ where each $a_i\in A$ and $\pi\in S_n$.  We describe an involution commuting with $g$ in the two cases $\pi = 1$ and $\pi \neq 1$.

Consider the case when $\pi=1$. Firstly, if some $a_j = 1$ then put $h := (c_1, \ldots, c_n)$ with $c_i = 1$ for $i \neq j$ and $c_j=x$ any involution $x \in A$. Then $h$ is an involution in $G$ commuting with $g$.

Alternatively, we have $a_j \neq 1$ for all $j$. As $g$ has prime order $p$ then each $a_i$ has order $p$. By the hypothesis, the $a_i$'s are contained in at most $n-1$ conjugacy classes of $A$ and thus,  there must be some $a_s$ and $a_t$ conjugate in $A$ with $s \neq t$. Choose $x \in A$ such that $a_t = a_s^x$. Without loss of generality suppose $s < t$. Let $\tau$ be the two cycle $(s~t) \in S_n$ and put $h := (c_1, \ldots, c_n)\tau$ with $c_s=x$, $c_t=x^{-1}$ and all other $c_i = 1$. Then $h$ is an involution and
$$\begin{array}{rl}
gh &=  (a_1, \ldots, a_s, \ldots, a_t, \ldots , a_n) (e_{A}, \ldots,c_s,\ldots,c_t, \ldots, e_{A})\tau  \\
&= (a_1, \ldots, a_t^{x^{-1}}x, \ldots, a_s^xx^{-1}, \ldots , a_n)\tau \\
&= (a_1, \ldots, xa_t, \ldots, x^{-1}a_s, \ldots , a_n)\tau \\
&= (e_{A_5}, \ldots,c_s, \ldots, c_t, \ldots, e_{A})\tau (a_1, \ldots, a_s, \ldots, a_t, \ldots , a_n) \\
&=hg.\end{array}$$

Now we resolve the second case, where $\pi \neq 1$. Since $e_G = g^p = (\ldots)\pi^p$ it follows that $\pi$ has order $p$.  Consider the case where $\pi$ is the single $p$-cycle $(1~2~\ldots~p)$. Then
$$1 = g^p = (a_1a_2\ldots a_p, a_2a_3\ldots a_pa_1, \ldots, a_pa_1a_2\ldots a_{p-1}, (a_{p+1})^p, \ldots, a_n^p)\pi^p$$
and in particular $a_1a_2\ldots a_p=1$. Let $x$ be an involution in $A$ and put $h := (x, x^{a_1}, x^{a_1a_2}, \ldots, x^{a_1a_2\ldots a_{p-1}}, 1, \ldots, 1)$.  Then
$$\begin{array}{rl}
gh &= (a_1, a_2, \ldots, a_n)\pi(x, x^{a_1}, x^{a_1a_2}, \ldots, x^{a_1a_2\ldots a_{p-1}}, 1, \ldots, 1) \\
&= (a_1x^{a_1}, a_2x^{a_1a_2}, \ldots,a_px, a_{p+1}, \ldots, a_n)\pi \\
&= (xa_1, x^{a_1}a_2, \ldots,a_px^{a_1a_2\ldots a_p}, a_{p+1}, \ldots, a_n)\pi \\
\multicolumn{2}{c}{\text{(as $a_1a_2\ldots a_p=1$)}} \\
&= (xa_1, x^{a_1}a_2, \ldots,x^{a_1a_2\ldots a_{p-1}}a_p, a_{p+1}, \ldots, a_n)\pi \\
&= (x, x^{a_1}, x^{a_1a_2}, \ldots, x^{a_1a_2\ldots a_p}, 1, \ldots, 1)(a_1, a_2, \ldots, a_n)\pi \\
&= hg.\end{array}$$
Also, since $x$ has order $2$ then so does $x^{a_1\ldots a_i}$ for all $i$ and so $h$ is an involution commuting with $g$.
When $\pi$ is arbitrary, we can construct an involution $h$ commuting with $g$  consisting of $k$ $p$-cycles in an analogous fashion by choosing  $h$ to have precisely $kp$ nontrivial entries and with these entries being suitable conjugates of $x$.
\end{proof}

The next two lemmas allow us to apply a result from the principal 1955 paper by Brauer and Fowler \cite{brauer} in the final proof of Proposition \ref{pWreathProduct}.
\begin{lemma}
\label{lConjugacyClassesOfInvolutions}
Let $A$ be a finite group of even order and $H \leq S_n$. For $n \geq 2$, the group $A \Wr H$ contains more than one class of involutions.
\end{lemma}
\begin{proof}
Since conjugation by elements of $A\Wr H$ preserves the number of entries of an element of $A^n$ that are nontrivial, the elements $(x,1,\ldots,1)$ and $(x,x,1,\ldots,1)$ for an involution $x$ in $A$ are nonconjugate involutions in $A\Wr H$.
\end{proof}

\begin{lemma}
\label{lWreathProductTrivialCentre}
If $A$ is a non-trivial group with trivial centre then $A \Wr H$ also has trivial centre for any group $H \leq S_n$.
\end{lemma}
\begin{proof}
Let $x := (a_1, a_2, \ldots, a_n)\pi$ be some element of $\gcentre{A \Wr H}$, with $a_i \in A$ and $\pi \in H \leq S_n$. Suppose that $\pi \neq e_H$. Then there is some index $j$ not fixed by $\pi$. Let $y=(y_1,\ldots,y_n)\in G^n$ such that $y_j$ is the only nontrivial entry of $y$. However, the only nontrivial entry of $y^x$ occurs in the coordinate given by $j^\pi\neq j$, and  so $\pi=1$. Since $A$ has trivial centre it follows that each $a_i$ is trivial and so $G\Wr H$ has trivial centre.
\end{proof}

Now we can put all the pieces together. We write $g \sim h$ when the group elements $g$ and $h$ commute.
\begin{proof}[Proof of Theorem \ref{pWreathProduct}]
Suppose $n \geq 2$ and set $G := A \Wr S_n$ for some $A$ satisfying the hypothesis. Let $g, h \in G \setminus \gcentre{G}$ and some $a \in \cent{G}{g}$ and $b \in \cent{G}{h}$ with $|a|$ and $|b|$ prime. By Lemma \ref{lPrimeOrderInvolutions}, there exist involutions $x, y \in G$ with $a \sim x$ and $b \sim y$. From Lemma \ref{lConjugacyClassesOfInvolutions} and Lemma \ref{lWreathProductTrivialCentre}, there are two classes of involutions in $G$ and $G$ has trivial centre. So, by a result from Brauer and Fowler \cite[Theorem 3D]{brauer}, the distance between the involutions $x$ and $y$ in $\cg{G}$ is at most $3$. Hence $\dist{g}{h} \leq 7$ in $\cg{G}$ and thus $\diam{\cg{G}} \leq 7$.
\end{proof}

An easy way to generate new groups is as a central product of smaller groups. We will prove Theorem \ref{tCentralProduct} in the remainder of this section, ruling out the taking of central products as a process to yield a counterexample to Iranmanesh and Jafarzadeh.

\begin{proof}[Proof of Theorem \ref{tCentralProduct}]
Let $G$ be a non-abelian finite group that is the central product of subgroups  $H$ and $K$, and take $g_1, g_2 \in G \setminus \gcentre{G}$. Then there exist $h_1, h_2 \in H$ and $k_1, k_2 \in K$ such that $g_1 = h_1k_1$ and $g_2=h_2k_2$. 

Consider first the case where $H$ and $K$ are both non-abelian. Since $g_1,g_2\notin \gcentre{G}$ and by observing that $\gcentre{G}=\gcentre{H}\gcentre{K}$, we must have that at least one of $h_1, k_1$ and at least one of $h_2, k_2$ are not in the centre of $H$ or $K$ appropriately, and thus not in the centre of $G$. Now without loss of generality there are 3 cases: (1) none of $h_1, h_2, k_1$ or $k_2$ are central. (2) $h_1$ is central and $h_2$ is not central (3) $h_1$ and $h_2$ are central.

Case 1: None of $h_1, h_2, k_1$ or $k_2$ are central. Then $g_1=h_1k_1 \sim h_1 \sim k_2 \sim h_2k_2=g_2$ is a path of non-central elements of length $3$ in $\cg{G}$ from $g_1$ to $g_2$. Moreover, if $\diam{\cg{H}}=2$ then there exists $x \in H \setminus \gcentre{H}$ such that $h_1 \sim x \sim h_2$. Hence $g_1 \sim x \sim g_2$ is a path of length $2$ between $g_1$ and $g_2$ in $\cg{G}$. Similarly, if $\diam{\cg{K}}=2$ we can also construct a path of length 2.

Case 2: $h_1$ is central and $h_2$ is not. Then $g_1=h_1k_1\sim h_2\sim h_2k_2=g_2$ is a path of non-central elements of length $2$.

Case 3: $h_1$ and $h_2$ are central. Then $g_1=h_1k_1 \sim h' \sim h_2k_2$ is a path of length $2$, where $h'$ is some non-central element of $H$ which exists since $H$ is non-abelian.

So in all possible cases there is a path of length at most $3$ between $g_1$ and $g_2$ in $\cg{G}$, and a path at most 2 if one of $\cg{G}$ or $\cg{K}$ has diameter 2. Therefore $\cg{G}$ is connected with diameter at most $\min\{3,\diam{\cg{H}},\diam{\cg{K}}\}$.

Consider instead the case where one of $H$ or $K$ is abelian. Without loss of generality we assume $K$ is abelian. Since $g_1$ and $g_2$ are non-central we must have $h_1, h_2 \notin \gcentre{H}$. Suppose $\cg{H}$ is connected, so there exists some path $h_1 \sim x_1 \sim \ldots \sim x_{d-1} \sim h_2$ where $x_i \in H \setminus \gcentre{H}$ with $d \leq \diam{\cg{H}}$. Then $g_1 \sim x_1 \sim \ldots \sim x_{d-1} \sim g_2$ is a path of length at most $d$ from $g_1$ to $g_2$ in $\cg{G}$. Therefore $\diam{\cg{G}} \leq \diam{\cg{H}}$. Conversely, suppose that $h_1$ and $h_2$ have distance $d$ in $\cg{H}$ and that $g_1=h_1k_1\sim a_1b_1\sim a_2b_2\sim \ldots a_\ell b_\ell \sim g_2=h_2k_2$ with $a_i\in H$ and $b_i\in K$. Then since each $a_ib_i$ is noncentral in $G$, each $a_i$ is noncentral in $H$ and also $h_1\sim a_1\sim a_2 \sim \ldots\sim a_\ell \sim h_2$ is a path in $\cg{H}$ of length $\ell$. Hence $\ell\geq d$ and so we must have $\diam{\cg{G}}=\diam{\cg{H}}$.  Similarly, when $\cg{H}$ is disconnected then $\cg{G}$ must also be disconnected.
\end{proof}

From [42, p. 26, Theorem 2.12], a finite nilpotent group is a direct product of its Sylow subgroups. So by applying Theorem \ref{tCentralProduct} we can conclude Corollary \ref{cNilpotentToPrime}.

\section{Relatively Small or Large Group Centres}
\label{sSizeOfGroupCentre}

First we prove Theorem \ref{tSmallCentre}.

\begin{proof}[Proof of Theorem \ref{tSmallCentre}]
Suppose $a$ and $b$ do not commute for some $a, b \in G \setminus \gcentre{G}$. Set $X := \cent{G}{a} \cap \cent{G}{b}$. If $|X| > |\gcentre{G}|$ then there exists some $x \in X \setminus \gcentre{G}$ and so $a \sim x \sim b$ is a path of length $2$ from $a$ to $b$ in $\cg{G}$. Thus we are just required to show that $|X| > |\gcentre{G}|$.

For each $g \in G$, define the map $\phi_g : G \rightarrow G'$ by $\phi_g(y) = [g,y]$. For any $h, k \in G$ we have $\phi_g(hk)  = g^{-1}k^{-1}h^{-1}ghk = g^{-1}k^{-1}gg^{-1}h^{-1}ghk=g^{-1}k^{-1}g[g,h]k$. Now $[g,h] \in G' \leq \gcentre{G}$ by the hypothesis, so $\phi_g(hk)=[g,h][g,k]=\phi_g(h)\phi_g(k)$. Therefore $\phi_g$ is a homomorphism with $\ker(\phi_g) =\{ y \in G \mid [g,y] = 1 \} = \cent{G}{g}$. By the First Isomorphism Theorem, $G / \ker(\phi_g) \cong \phi_g(G) \leq G' \leq \gcentre{G}$. Hence $|G| / |\cent{G}{g}| \leq |G'| \leq |\gcentre{G}|$. Rearranging gives $|G| / |\gcentre{G}| \leq |\cent{G}{g}|$ for any $g \in G$. As $|\gcentre{G}|^3 < |G|$, it follows that $|\gcentre{G}|^2 < |\cent{G}{g}|$.

Let $\phi_b'$ be the restriction of $\phi_b$ to $\cent{G}{a}$. Then $\ker(\phi_b') = \cent{G}{a} \cap \ker(\phi_b) = X$. By the First Isomorphism Theorem, $\cent{G}{a} / X  \cong \phi_b'(\cent{G}{a}) \leq G' \leq \gcentre{G}$. Taking cardinalities and rearranging $|X| \geq |\cent{G}{a}| / |\gcentre{G}|$. We showed in the previous paragraph that $|\gcentre{G}|^2 < |\cent{G}{a}|$ and so $|X| > |\gcentre{G}|^2 / |\gcentre{G}| = |\gcentre{G}|$. The result follows.
\end{proof}

Now we consider cosets of $\gcentre{G}$ to prove Theorem \ref{tPrimeCentre}. As observed in \cite{vahidi}, it is clear that if $x\sim y$ in $G$ then $g \sim h$ for all $g \in x\gcentre{G}$ and $h \in y\gcentre{G}$. We introduce the following notation. For any subgroup $H$ of the quotient group $G/ \gcentre{G}$, we use $\overline{H}$ to denote all  the non-central elements of $G$ contained in elements of $H$. That is,
$$\overline{H} := \bigcup_{h\gcentre{G} \in H\setminus \{\gcentre{G}\}} h\gcentre{G},$$
so $\overline{H} \subseteq G\setminus \gcentre{G}$. A \emph{clique} in a graph is a subset of vertices all pairwise adjacent. Now we can find large cliques in the commuting graphs of groups with non-trivial centres.

\begin{lemma}
\label{lTwoCosets}
For all $x, y \in G \setminus \gcentre{G}$, if $x \sim y$ then $\overline{\langle x \gcentre{G},  y \gcentre{G}\rangle}$ is a clique in $\cg{G}$. In particular, $\overline{\langle x \gcentre{G}\rangle}$ is a clique in $\cg{G}$.
\end{lemma}
\begin{proof}
Since $x \sim y$, we have $\langle x \gcentre{G},  y \gcentre{G}\rangle = \{ x^iy^j \gcentre{G} | i, j \in \mathbb{Z}\}$. Moreover $x^iy^j \sim x^ky^\ell$ for all $i, j, k, \ell \in \mathbb{Z}$ and so any two elements of $G$ contained in an element of $\langle x \gcentre{G},  y \gcentre{G}\rangle$ commute and the result follows. 
\end{proof}

\begin{proof}[Proof of Theorem \ref{tPrimeCentre}]
Suppose $|G : \gcentre{G}|$ is a product of at most three primes, not necessarily distinct.  The index of $\gcentre{G}$ in $G$ can never be prime (for example, see \cite[Theorem 9.3]{gallian}) and so $|G:\gcentre{G}|$ is divisible by at least two primes~$p,q$.

Assume that $\cg{G}$ is connected. By Cauchy's Theorem, there exists some $a \in G \setminus \gcentre{G}$ such that $|a\gcentre{G}|=p$. Set $A := \langle a\gcentre{G} \rangle$. Then by Lemma \ref{lTwoCosets}, $\overline{A}$ is a clique of $\cg{G}$. Since $\cg{G}$ is connected, there exists some $x \in \overline{A}$ and $b \in G \setminus \left(\overline{A} \cup \gcentre{G}\right)$ such that $x \sim b$. Since $|a\gcentre{G}|$ is prime and $x\gcentre{G} \in A \setminus \{\gcentre{G}\}$ it follows that $\langle x \gcentre{G}\rangle = \langle a \gcentre{G} \rangle$. By Lemma \ref{lTwoCosets}, $b \sim x$ implies $b \sim a$ and  $\overline{B}$ is a clique, where $B := \langle a \gcentre{G}, b \gcentre{G} \rangle$.

Since  $b\gcentre{G} \notin A$ we have $|A| < |B|$.  If $B=G/\gcentre{G}$ then  $a \sim g$ for all $g \in G\backslash \gcentre{G}$, contradicting  $a \notin \gcentre{G}$. Thus by Lagrange's Theorem we may assume that $|B|=pq$ and $|G:\gcentre{G}|=pqr$ for some prime $r$.

Now, since $\cg{G}$ is connected, there exists some $y \in \overline{B}$ and $c \in G \setminus\left(\overline{B} \cup \gcentre{G} \right)$ such that $y \sim c$. Let $C := \langle a \gcentre{G}, b \gcentre{G}, c \gcentre{G} \rangle$. Then $B < C$ and since $B$ has prime index in $G/\gcentre{G}$ it follows that $C=G/\gcentre{G}$. Therefore $\overline{C}$ contains all of the elements of $\cg{G}$. Since $y \in \overline{B}$ and $\overline{B}$ is a clique of $\cg{G}$, it follows that $y$ commutes with $a$ and $b$. As $y$ also commutes with $c$, we have that $y$ commutes with every coset leader in $C$ and thus by Lemma \ref{lTwoCosets} commutes with every element of those cosets. So $y \sim g$ for all $g \in \overline{C}$ and hence $y \in \gcentre{G}$. This is a contradiction and so the commuting graph of $G$ must be disconnected.
\end{proof}

\section{Illustrative Examples}
\label{sMotivatingExamples}
Theorem \ref{tSmallCentre}  gave bounds on the diameters of the commuting graphs of certain groups with nilpotency class $2$. Our first example here shows that an increase in nilpotency class does not necessarily result in an increase in commuting graph diameter.

We construct an example that shows for all primes $p$ and integers $c \geq 3$, there exists a $p$-group with nilpotency class $c$ and commuting graph diameter $3$. Denote the group of $n\times n$ invertible matrices over the field $\Z_p$ of integers modulo $p$ by $\GL(n,p)$ and denote its subgroup consisting of all lower unitriangular matrices  by $\ULT(n,p)$, that is, all lower triangular matrices whose entries along the main diagonal are all $1$.   Then $|\ULT(n,p)|=p^{n(n-1)/2}$. 

Denote the $r\times r$ identity matrix by $I_r$, the $r\times s$ zero matrix by $0_{r,s}$, and the matrix with $1$ in entry $i,j$ and $0$ elsewhere by $E_{i,j}$. Weinstein shows that the centre of this group is given by $\gcentre{\ULT(n,p)} = \{ I_n + aE_{n,1} | a \in \Z_p \}$ \cite[Theorem 3.2.2]{weinstein} and that $\ULT(n,p)$ has nilpotency class $n-1$ \cite[Theorem 3.2.3]{weinstein}. 

\begin{proposition}
\label{pUnitriangular}
For a prime $p$, the group $\ULT(3,p)$ has disconnected commuting graph while for $n\geq 4$ the commuting graph of $\ULT(n,p)$ has diameter $3$. 
\end{proposition}
\begin{proof}
When $n=3$ we have $|\ULT(n,p) : \gcentre{\ULT(n,p)}|=p^2$ and so Theorem \ref{tPrimeCentre} tells us that $\cg{\ULT(n,p)}$ is disconnected. Suppose instead that $n\geq 4$.

Consider the subset 
$$\mathcal{X} = \left\{\begin{array}{c|c}I_n+xE_{n-1,1}+yE_{n,2} & x,y \in \Z_p \text{ and at least one of }x, y \neq 0 \end{array}\right\}$$
of $\ULT(n,p)$. This set is nonempty, does not intersect $\gcentre{\ULT(n,p)}$ and, as $n \geq 4$, its elements pairwise commute.

Take some arbitrary 
$$A=\begin{bmatrix}
1 & 0 & 0 & \ldots &  0 \\
a_{2,1} & 1 & 0 & \ldots & 0 \\
a_{3,1} & a_{3,2} & 1 &  \ldots &  0 \\
\vdots & \vdots & \ddots & \ddots & \vdots \\
a_{n, 1} & a_{n,2} & a_{n,3}&\ldots \,\, a_{n,n-1} & 1
\end{bmatrix} \in \ULT(n,p) \setminus \gcentre{\ULT(n,p)}$$
and a variable $X=I_n+xE_{n-1,1}+yE_{n,1}\in \mathcal{X} \subseteq \ULT(n,p) \setminus \gcentre{\ULT(n,p)}.$ We can find values for $x$ and $y$ so that $A \sim X$.
As $n \geq 4$ we have
$$AX = \begin{bmatrix}
1&0&\ldots&0&0 \\
a_{2,1}&1&\ldots&0&0\\
\vdots & \vdots & \ddots & \vdots & \vdots \\
a_{n-1,1}+x & a_{n-1,2} & \ldots & 1 & 0 \\
a_{n,1} + x a_{n,n-1} & a_{n,2} + y & \ldots & a_{n,n-1} & 1
\end{bmatrix}$$
and
$$XA = \begin{bmatrix}
1&0&\ldots&0&0 \\
a_{2,1}&1&\ldots&0&0\\
\vdots & \vdots & \ddots & \vdots & \vdots \\
a_{n-1,1}+x & a_{n-1,2} & \ldots & 1 & 0 \\
a_{n,1} + y a_{2,1} & a_{n,2} + y & \ldots & a_{n,n-1} & 1
\end{bmatrix}.$$
If $a_{2,1} \neq 0$ then set $x := 1, y := \frac{a_{n,n-1}}{a_{2,1}}$ and it follows that $AX=XA$. Otherwise $a_{2,1} = 0$ and setting $x := 0, y := 1$ yields $AX=XA$. Therefore, for all $A \in \ULT(n,p) \setminus \gcentre{\ULT(n,p)}$ there exists $X \in \mathcal{X}$ such that $A$ commutes with $X$. Suppose $A, B \in \ULT(n,p) \setminus \gcentre{\ULT(n,p)}$. Then there exists $X, Y \in \mathcal{X}$ with $A \sim X$ and $B \sim Y$. Since the elements of $\mathcal{X}$ pairwise commute,  $A \sim X \sim Y \sim B$ is a path of length $3$ between $A$ and $B$ in $\cg{\ULT(n,p)}$. Thus $\cg{\ULT(n,p)}$ is connected with $\diam{\cg{\ULT(n,p)}} \leq 3$.

Next we construct two elements of $\ULT(n,p)$ whose centralisers intersect only on the centre of $\ULT(n,p)$ to show that $\diam{\cg{\ULT(n,p)}} = 3$. Put 
$$A := I_n+\begin{bmatrix}0_{1\times n-1} & 0 \\ I_{n-1} & 0_{n-1\times 1} \end{bmatrix}$$
 and $B := I_n+E_{2,1}$.

Take some $X =(x_{i,j}) \in \ULT(n,p)$ commuting with $A$. Examining the proof of  \cite[Lemma 2.4]{giudici2010}
we can see that  for any $k \in \{1, \ldots, n\}$ we have that $x_{k+i, 1+i} = x_{k,1}$ for all $i \in \{1, \ldots, n-k\}$. That is, all the entries on any given top-left to bottom-right diagonal of $X$ below the main diagonal are equal, not just those along the main diagonal.

Suppose that $X$ also commutes with $B$. Set $T := XB = BX = (t_{i,j})$. Consider the first column of $T$; for any $k \geq 3$ evaluating $BX$ gives $t_{k,1} = x_{k,1}$ and evaluating $XB$ yields $t_{k,1} = x_{k,1} + x_{k,2}$. Equating these values for $t_{k,1}$ yields $x_{k,1} = x_{k,1} + x_{k,2}$ and cancelling gives $x_{k,2} = 0$ for all $k \geq 3$. Since $X$ commutes with $A$, all the elements on the diagonal of $X$ containing $x_{k,2}$ are equal and so are all zero for $k \geq 3$. Thus every diagonal of $X$ contains all zero entries except for the main diagonal and $x_{n,1}$. So $X \in \gcentre{\ULT(n,p)}$. Therefore $\cent{\ULT(n,p)}{A} \cap \cent{\ULT(n,p)}{B} = \gcentre{\ULT(n,p)},$ which implies $\dist{A}{B} \geq 3$ in $\cg{\ULT(n,p)}$ and the result follows.
\end{proof}

\begin{wraptable}{r}{0.4\textwidth}
\label{tGroupsDiam6}
\begin{center}
\begin{tabular}{|l|r|}
\hline
Order & Number \\
\hline
1152 & 157451 \\
1152 & 157452 \\
1176 & 90 \\
1176 & 91 \\
1176 & 92 \\
1176 & 95 \\
1176 & 96 \\
1176 & 97 \\
1176 & 214 \\
1500 & 115 \\
1728 & 47862 \\
1944 & 2289 \\
1944 & 2290 \\
\hline
\end{tabular}
\end{center}
\caption{The order and small group database number of all the groups with commuting graph diameter $6$ of order not equal to $1024$ or $1536$ and at most $2000$.}
\label{tab:diam6}
\end{wraptable}

Calculating the diameter of a graph takes a number of computations proportional to the number of vertices of the graph cubed. Vahidi and Talebi \cite{vahidi} observed that the diameter of the commuting graph of a group is equal to the diameter of the subgraph induced by a transervsal of the center. Since most groups have non-trivial centre, this observation greatly speeds up a computer program to calculate the diameter of the commuting graph of many groups. Using this optimisation, we have used the \textsc{Magma}\cite{magma} implementation of the small group database \cite{besche2001, besche} to calculate the diameter of the commuting graphs of all groups with order up to 2000 except those with order 1024 and 1536. The largest commuting graph diameter found was $6$, and only $13$ such groups were found.  Table \ref{tab:diam6} gives the order and number in the small groups database of all of the $13$ groups found with commuting graph diameter $6$.

The subgroup of $\GL(3,7)$ generated by the matrices 
$$\left[
\begin{array}{ccc}
3&6&2\\
2&0&1\\
0&0&1\\
\end{array}
\right] \textrm{ and } \left[
\begin{array}{ccc}
0&4&1\\
5&0&3\\
0&0&1\\
\end{array}
\right]$$
 has commuting graph diameter $6$ and is isomorphic to the group numbered $214$ of order $1174$ in the small groups database. Inspired by this group discovered through our computer search, we will now construct an infinite family of groups with trivial centre and commuting graph having diameter $6$.

\begin{construction}\label{conW}
Let $p$ be a prime of the form $p=3n+1$ for some integer $n$. Let $\alpha$ be an element of order $3$ in $\Z_p^*$, the multiplicative group of $\Z_p$, and set $\be := \al + 1$. Then $\al\neq 1, \al^3 = 1, \be \neq -1, \be^3 = -1$ and $\al^2 + \be^2=-1$.  In $\GL(2,p)$, let $I$ be the identity matrix and take
$$
J := \begin{bmatrix}0&1\\-1&0\end{bmatrix},
K := \begin{bmatrix}\al&\be\\\be&-\al\end{bmatrix},
L := \begin{bmatrix}\be&-\al\\-\al&-\be\end{bmatrix} \in \GL(2,p).$$
Note that $J$, $K$ and $L$ are elements of order 4 squaring to $-I$ and so $\mQ := \langle J, K, L \rangle \cong Q_8$.

Let
$$Z := \begin{bmatrix}-\be&\al\\0&1\end{bmatrix} \in \GL(2,p)$$ and note that $Z$ is an element of order 3 such that $J^Z=K$, $K^Z=L$ and $L^Z=J$. Thus  $\mS := \langle J,K,L,Z \rangle$ is a subgroup of $\GL(2,p)$ isomorphic to $\SL(2,3)$.

Let  $\W := \Z_p^2\rtimes \mS\leqslant \AGL(2,p)$.   Let $x\in \W$. Then we can write $x$ as $(\bx,X)$ for some $\bx \in \Z_p^2$ and $X \in \mS$. For another element $y = (\by,Y)\in \W$ we have
\begin{equation}
\label{wproduct}
xy=(\bx+\by X^{-1}, XY)
\end{equation}
and $x$ has inverse
\begin{equation}
\label{winverse}
x^{-1}=(-\bx X, X^{-1}).
\end{equation}
We identify $\mS$ and $\Z_p^2$ with the subgroups $\{((0,0),X) | X \in \mS\}$ and $\{(\bx,I)| \bx \in \Z_p^2\}$ of $\W$ respectively. When $\bx =(0,0)=e_{\Z_p^2}$ in $x=(\bx,X)$ then sometimes we omit the $\bx$ and the ordered pair and unambiguously write $x=X$.
\end{construction}

We note that the group $\W$ in Construction \ref{conW} has order $2^3.3p^2$ and is soluble as it is the semidirect product of two soluble groups.

We will see in Lemma \ref{threetop} that $Z$ has the important property of fixing $(0,1) \in \Z_p$ which precipitates $\cg{\W}$ being connected. 

The goal of the remainder of this section is to prove the following theorem:
\begin{theorem}
\label{diam6}
The commuting graph of the group  $\W$ in Construction \ref{conW} has diameter $6$.
\end{theorem}

In order to highlight why this theorem is true, we describe the relevant properties of $\W$ in a series of lemmas.

\begin{lemma}
\label{centnegi}
$\cent{\W}{-I} = \mS$.
\end{lemma}
\begin{proof}
Clearly  $\mS \leq \cent{\W}{-I}$. Take some $a \in \cent{\W}{-I}$, so $a = (\ba,A)$ for some $\ba \in \Z_p^2$ and $A \in \mS$. Then $(-I)a=a(-I) = (\bbb,B)$ for some $\bbb \in \Z_p^2$ and $B \in \mS$. Using equation \eqref{wproduct}, evaluating $(-I)a$ we see that $\bbb = -\ba$ and evaluating $a(-I)$ we have $\bbb=\ba$.  Since $p$ is odd, it follows that  $\ba=(0,0)$, consequently $a \in \mS$. Hence $\cent{\W}{-I} \leq \mS$ and in fact equality holds.
\end{proof}

In this next lemma we discover that involutions in $\W$ do not commute with elements of the subgroup $\Z_p^2 \lhd \W$.
\begin{lemma}
\label{order2}
For any involution $g \in \W$ we have $\cent{\W}{g} \cong \mS$ and so $|\cent{\W}{g}|=24$.
\end{lemma}
\begin{proof}
Let $g$ be some involution in $\W$. By Sylow's Theorems, $\langle g \rangle$ is contained in some Sylow $2$-subgroup $H$ of $\W$
 and there exists some $x \in \W$ such that $H^x = \mQ$. Thus $g^x$ is an involution in $\mQ$.  Since $-I$ is the only involution in $\mQ\cong Q_8$,  Lemma \ref{centnegi} reveals that $\cent{\W}{g} \cong \cent{\W}{-I} = \mS$ and so $|\cent{\W}{g}|=|\mS|=24$.
\end{proof}

We collect some easily verifiable facts about $\SL(2,3)$.
\begin{lemma}
\label{sl23classes}
~
\begin{itemize}
 \item There are only two conjugacy classes of elements of order $6$ in $\SL(2,3)$, and if $g$ is an element of order $6$ then $g$ and $g^{-1}$ lie in different conjugacy classes. 
 \item There is only one conjugacy class of elements of order $4$ in $\SL(2,3)$. 
 \item If $g\in\SL(2,3)$ has order $4$ or $6$ then $\cent{\SL(2,3)}{g}=\langle g\rangle$.
\end{itemize}
\end{lemma}

For the proof of Lemma \ref{order46} below, we denote the conjugacy class of an element $g$ in a group $G$ by $\cl{G}{g}$.
\begin{lemma}
\label{order46}
If $g\in \W$ has order $4$ or $6$ then  $\cent{\W}{g} = \langle g \rangle$ and as an immediate consequence there are no elements of order $8, 12$ or $24$ in $\W$.
\end{lemma}
\begin{proof}
First we determine $\cent{\W}{g}$ when $|g|=4$. It turns out that $g$ is conjugate to $J \in \mQ$, whose centraliser in $\W$ can be shown to have order $4$. Indeed, by Sylow's Theorems,  $g$ is contained in some Sylow $2$-subgroup $H$ of $\W$ and hence there exists $x \in \W$ such that $g^x\in \mQ\cong Q_8$. By Lemma \ref{sl23classes}, all elements of order 4 in $\mS$ are conjugate so 
$g^x \in \cl{\W}{J}$. Hence $\cent{\W}{g} \cong \cent{\W}{J}$ and it remains to show $\cent{\W}{J}=\langle J\rangle$.

Suppose some $x \in \W$ commutes with $J=\begin{bmatrix}0&1\\-1&0\end{bmatrix}$. Then $x = ((x_1,x_2),X)$ for some $(x_1,x_2) \in \Z_p^2$ and $X \in \mS$. Let $a := xJ = Jx = (\ba,A)$ for some $\ba \in \Z_p^2$ and $A\in \mS$. Evaluating $Jx$ we see that $\ba=(x_1,x_2)J^{-1}=(x_2,-x_1)$ and evaluating $xJ$ gives $\ba=(x_1,x_2)$. So $x_2=x_1$ and $-x_1=x_2$. Substituting and rearranging gives $2x_1=0$ which, as $p$ is odd, implies that $x_1=x_2=0$. So $x \in \mS$ and $\cent{\W}{J}=\cent{\mS}{J}=\langle J\rangle$, by Lemma \ref{sl23classes}.

What happens instead when $|g|=6$? In this case, $\langle g \rangle$ is a $\{2,3\}$-subgroup of $\W$. Since $\W$ is soluble,  Hall's Theorems imply that there exists some Hall $\{2,3\}$-subgroup $H$ of $\W$ containing $\langle g \rangle$. Moreover, all Hall $\{2,3\}$-subgroups are conjugate in $\W$. The group $\mS$ has index $p^2$ in $\W$, so is a Hall $\{2,3\}$-subgroup of $\W$. Thus there exists some $x \in \W$ such that $H^x \leq \mS$ and $g^x$ is some element of order $6$ in $\mS$. By Lemma \ref{sl23classes}, there are only two conjugacy classes of order $6$ in $\mS$ and these classes have representatives $JZ^{-1}$ and $(JZ^{-1})^{-1}=ZJ^{-1}$. Hence $g^x$ is conjugate to $JZ^{-1}$ or $ZJ^{-1}$, and note that $\cent{\W}{JZ^{-1}}=\cent{\W}{ZJ^{-1}}$.

Take any $y \in \cent{\W}{JZ^{-1}}$. So $y=((y_1,y_2),Y)$ for some $(y_1,y_2) \in \Z_p^2$ and $Y \in \mS$. Set $a := y(JZ^{-1})=(JZ^{-1})y= (\ba, A)$ for some $\ba \in \Z_p^2$ and $A\in S$. Evaluating $y(JZ^{-1})$ by equation \eqref{wproduct} we see that $\ba = (y_1,y_2)$ and evaluating $(JZ^{-1})y$ gives $\ba=(\al y_1+y_2, \be y_2)$. Equating these two expressions for $\ba$ yields $y_1=\al y_1+y_2$ and $y_2=-\be y_1$.  Hence $y_1=y_1(\al-\be)=-y_1$, since $\be=\al+1$.  Thus $y_1=0=y_2$. Therefore $y \in \mS$ and $\cent{\W}{JZ^{-1}} =\cent{\mS}{JZ^{-1}}=\langle JZ^{-1}\rangle$, by Lemma \ref{sl23classes}.

 Suppose that there exists an element $h \in \W$ of order $8$. Then $\langle h \rangle$ is an abelian group so $h^2$ commutes with all $8$ elements of $\langle h \rangle$ in $\W$. But $h^2$ has order $4$ and so only commutes with $4$ elements in $\W$, a contradiction. Hence there are no elements of order $8$ in $\W$. Elements of order $12$ or $24$ in $\W$ similarly give rise to elements of order $6$ commuting with more than $6$ elements of $\W$.
\end{proof}

\begin{lemma}
\label{trivialcentre}
The centre of $\W$ is trivial.
\end{lemma}
\begin{proof}
By Lemma \ref{order46}, an element $g$ of order $4$ in $\W$  has centraliser of order $4$ and so is not central. However, $g$ commutes with every central element of $\W$, and so a nontrivial central element of $\W$ has order $2$. But, by Lemma \ref{order2}, every involution of order $2$ in $\W$ commutes with only $24$ elements of $\W$ and so is non-central. Therefore $\gcentre{\W} = \{e\}$.
\end{proof}

Now we give two involutions in $\W$ that have distance at least $3$ in $\cg{\W}$.
\begin{lemma}
\label{noncominvolutions}
Let $a := ((1,1),-I) \in \W$. Then $\cent{\W}{-I}\cap \cent{\W}{a} = \{e\}$.
\end{lemma}
\begin{proof}
Take any $x \in \cent{\W}{-I}\cap \cent{\W}{a}$ with $x=(\bx,X)$ for some $\bx \in \Z_p^2$ and $X \in \mS$. By Lemma \ref{centnegi}, $x \in \mS$ so $\bx = (0,0)$ and $x = X$. Since $X$ commutes with $a=((1,1),-I)$, it follows that $(1,1)=(1,1)X^{-1}$. Any power of $X^{-1}$ also fixes $(1,1)$. Observe that $-I$ is the only element of order $2$ in $\mS$ and $-I$ does not fix $(1,1)$. So no power of $X^{-1}$ is $-I$ and hence the order of $X^{-1}$ must be odd. As $|\mS|=2^3.3$, the order of $X^{-1}$ must then be $3$ or $1$. From \cite{mackiw}, there are $8$ elements of order $3$ in $\SL(2,3)\cong\mS$. The $8$ elements
$$Z,JZ,KZ,LZ,Z^{-1},J^{-1}Z^{-1},K^{-1}Z^{-1},L^{-1}Z^{-1} \in \mS$$
have order $3$ and so  $X^{-1}$  must be one of them. However, none of them fix $(1,1)$, so instead the only possibility is that $X=I$ and $x=e_\W$.
\end{proof}

Now we give a lower bound on the diameter of $\cg{\W}$ by finding two elements of order $4$ in $\W$ with distance between them at least $6$ in $\cg{\W}$.
\begin{lemma}
\label{diam6lower}
$6 \leq \diam{\cg{\W}}$.
\end{lemma}
\begin{proof}
Put $a := ((1,1),-I)$ and $b := -I$. We will see that $\dist{a}{b} \geq 4$ in $\cg{\W}$, then deduce that an element of order $4$ commuting with $a$ has distance at least $6$ in $\cg{\W}$ from an element of order $4$ commuting with $b$.

From Lemma \ref{noncominvolutions}, no non-trivial elements of $\W$ commute with both $a$ and $b$. Now assume that $\dist{a}{b} = 3$, that is, that there exists non-trivial $x \in \cent{\W}{a}$ and $y \in \cent{\W}{b}$ such that $x \sim y$. In Lemma \ref{order2} we saw that $|\cent{\W}{a}| = |\cent{\W}{b}| = 24$ and so $|x|$ and $|y|$ divide $24$. Moreover, from Lemma \ref{order46}, there are no elements of orders $8, 12$ or $24$ in $\W$ so $|x|$ and $|y|$ are each one of $2,3,4$ or $6$. We will examine these different possibilities for the orders of $x$ and $y$ in turn.

Since $a$ and $b$ are involutions and by Lemma \ref{order2} are the unique involution in their respective centralisers, it follows that neither $x$ nor $y$ are involutions. Suppose that one of $|x|$ and $|y|$ is $4$ or $6$. Without loss of generality take $|x|=4$ or $6$.
Then by Lemma \ref{order46} the subgroup $\cent{\W}{x}$ is cyclic. As $y \in \cent{\W}{x}$ we have $y=x^k$ for some integer $k$, so $y\in \cent{\W}{a}\cap \cent{\W}{b}$, contradicting Lemma \ref{noncominvolutions}.

The only remaining possibility is that $|x| = |y| = 3$. Since $x$ and $y$ commute, every element of $\langle x, y \rangle$ can be written as $x^sy^t$ for some integers $0 \leq s,t\leq 2$ and so $|\langle x, y \rangle| \leq 9$. Also $3$ divides $|\langle x,y \rangle|$  which in turn divides $|\W|=2^3.3p^2$, where $p \neq 3$. Thus $|\langle x,y\rangle| = 3$ and $\langle x,y\rangle = \langle x\rangle$. Therefore $y \in \langle x\rangle$ and $y=x^k$ for some integer $k$. Now, as $x \in \cent{\W}{a}$, it follows that $y\in \cent{\W}{a}\cap \cent{\W}{b}$, contradicting Lemma \ref{noncominvolutions}.

For all possible values of $|x|$ and $|y|$ we have derived a contradiction; thus no such $x$ and $y$ can exist and $\dist{a}{b} \geq 4$ in $\cg{\W}$. Now by Lemma \ref{order2} we have that $\cent{\W}{a} \cong \cent{\W}{b} \cong \mS$. There are elements of order $4$ in $\mS$, for example $J$, so there exists some $q \in \cent{\W}{a}$ and $r\in \cent{\W}{b}$ with $|q|=|r|=4$. By Lemma \ref{trivialcentre}, the centre of $\W$ is trivial so $a,b,q$ and $r$ are all vertices of $\cg{\W}$.

Lemma \ref{order46} indicates that $\cent{\W}{q} = \langle q \rangle = \langle q^3 \rangle$ so the only non-trivial elements that commute with $q$ and $q^3$ are $q, q^2=a$ and $q^3$. As $\dist{a}{b} \geq 4$ neither $q$ nor $q^3$ equal $r$. So if a minimal path from $q$ to $r$ exists in $\cg{\W}$ then it must have $a$ as its second vertex. Similarly it must have $b$ as its penultimate vertex and therefore such a path must have length at least $2+\dist{a}{b}\geq 6$. Hence $6 \leq \diam{\cg{\W}}$.
\end{proof}

To complete the proof of  Theorem \ref{diam6} we need to show that the diameter is at most 6. The following lemmas enable us to find paths of length at most $6$ between any two non-central elements of $\W$.
\begin{lemma}
\label{threetop}
For $g \in \W$ with $|g|=3$, there exists an element of order $p$ in $\W$ commuting with $g$.
\end{lemma}
\begin{proof}
 Firstly observe that  $((0,1),I)\in \W$ has order $p$ and commutes with the element $Z=\begin{bmatrix}-\be&\al\\0&1\end{bmatrix}$ of order $3$ in $\W$. By Sylow's Theorem, $g$ is conjugate to $Z$ or $Z^{-1}$, and so the result follows.
\end{proof}

\begin{lemma}
\label{twotop}
Any involution in $\W$ has distance at most $2$ from $((0,1),I)$ in $\cg{\W}$.
\end{lemma}
\begin{proof}
Let $a$ be an involution in $\W$. Then $a = ((a_1,a_2),A)$ for some $a_1, a_2 \in \Z_p$ and $A \in \mS$ with $A^2=I$. We must have $A \neq I$ as otherwise $|a|=|(a_1,a_2)|=p$~or~$1$. Thus $A = -I$ as this is the only involution in $\mS$. We can see now that $((a_1(1+\be)/2\be,-a_1\al/2\be),Z) \in \W$ is a non-trivial element that commutes with both $a$ and $((0,1),I)$. 
\end{proof}

\begin{lemma}
\label{orderp}
All elements of order $p$ in $\W$ commute.
\end{lemma}
\begin{proof}
The abelian subgroup $\{ (\bx,I) | \bx \in \Z_p^2 \} \leq \W$ is normal in $\W$ and so is the unique Sylow $p$-subgroup of $W$. Therefore it contains all the elements of order $p$ in $\W$.
\end{proof}

It is relatively easy to now construct a path of length at most $6$ between any two vertices of $\cg{\W}$.
\begin{lemma}
\label{diam6upper}
$\diam{\cg{\W}}\leq 6$.
\end{lemma}
\begin{proof}
Set $w := ((0,1),I) \in \W \setminus \gcentre{\W}$. We will show that any $g \in \W\setminus \gcentre{\W}$ has distance at most $3$ from $w$ in $\cg{\W}$. A path of length at most $6$ between any two non-central elements $g,h$ of $\W$ can then be constructed by taking at most three steps from $g$ to $w$ and then at most three step from $w$ to $h$.

For the following calculations, recall by Lemma \ref{trivialcentre} that any non-trivial element of $\W$ is non-central and so in $\cg{\W}$. Let $g$ be any non-trivial element in $\W$ and let $x$ be an element of prime order in the non-trivial group $\cent{\W}{g}$. Then $|x|$ is one of $2, 3$ or $p$. If $|x|=p$ then, by observing that also $|w|=p$ and applying Lemma \ref{orderp}, we have $x \sim w$ and $g \sim x \sim w$ is a path of length $2$ from $g$ to $w$ in $\cg{\W}$. If $|x|=3$ then, by Lemma \ref{threetop}, there exists some $y \in \W$ of order $p$ commuting with $x$ and $g \sim x \sim y \sim w$ is a path of length $3$ from $g$ to $w$ in $\cg{\W}$. The remaining possibility is that $|x| = 2$. Then Lemma \ref{twotop} tells us that $\dist{x}{w} \leq 2$ in $\cg{\W}$ and it follows that $\dist{g}{w} \leq 3$.
\end{proof}

The proof of Theorem \ref{diam6} is concluded by combining Lemmas \ref{diam6lower} and \ref{diam6upper}. By Dirichlet's Theorem \cite[Theorem 10.9]{nathanson}, there are infinitely many choices for the prime $p$ so we obtain infinitely many groups with trivial centre and commuting graph having diameter $6$.

%\bibliographystyle{plain}
%\bibliography{ref}

\end{document}